\documentclass[12pt]{article}
\usepackage{amssymb}
\usepackage{amsmath}
\usepackage{amsfonts}
\usepackage{color}
\usepackage{hyperref}
\usepackage[normalem]{ulem}

\providecommand{\U}[1]{\protect\rule{.1in}{.1in}}
{\setlength\paperheight {305mm}
\setlength\paperwidth  {210mm}}

\providecommand{\U}[1]{\protect\rule{.1in}{.1in}}

\newtheorem{theo}{Theorem}
\newtheorem{lem}[theo]{Lemma}
\newtheorem{prop}[theo]{Proposition}
\newtheorem{cor}[theo]{Corollary}
\newtheorem{rem}{Remark}

\newenvironment{proof}[1][Proof]{\noindent \textbf{#1.} }{\ \rule{0.5em}{0.5em}}
\newenvironment{dem}[1][Proof]{\noindent \textbf{#1.} }{\ \rule{0.5em}{0.5em}}
\ifx\pdfoutput\relax\let\pdfoutput=\undefined\fi
\newcount\msipdfoutput
\ifx\pdfoutput\undefined\else
\ifcase\pdfoutput\else
\ifx\paperwidth\undefined\else
\ifdim\paperheight=0pt\relax\else\pdfpageheight\paperheight\fi
\ifdim\paperwidth=0pt\relax\else\pdfpagewidth\paperwidth\fi
\fi\fi\fi
\input{tcilatex}
\begin{document}

\title{Convex regularization and subdifferential calculus\thanks{%
The first author was partially supported by Centro de Modelamiento Matem\'{a}%
tico (CMM), ACE210010 and FB210005, BASAL funds for center of excellence and
ANID-Chile grant Fondecyt Regular 1240335. The research of the second and
third author is supported by Grant PID2022-136399NB-C21 funded by
MICIU/AEI/10.13039/501100011033 and by ERDF/EU. The research of the third
author is also supported by MICIU of Spain and Universidad de Alicante
(Contract Beatriz Galindo BEA- GAL 18/00205),\ AICO/2021/165 of Generalitat
Valenciana.}}
\author{Rafael Correa\thanks{%
e-mail: rcorrea@dim.uchile.cl} \\
{\small Universidad de Chile, Chile} \and Abderrahim Hantoute\thanks{%
e-mail: hantoute@ua.es} \\
{\small Universidad de Alicante, Spain} \and Marco A. L\'{o}pez\thanks{%
e-mail: marco.antonio@ua.es} \\
{\small Universidad de Alicante, Spain}}
\date{}
\maketitle

\begin{abstract}
This paper deals with the regularization of the sum of functions defined on
a locally convex space through their closed-convex hulls in the bidual
space. Different conditions guaranteeing that the closed-convex hull of the
sum is the sum of the corresponding closed-convex hulls are provided. These
conditions are expressed in terms of some $\varepsilon $-subdifferential
calculus rules for the sum. The case of convex functions is also studied,
and exact calculus rules are given under additional
continuity/qualifications conditions. As an illustration, a variant of the
proof of the classical Rockafellar theorem on convex integration is proposed.


\end{abstract}

\section{Introduction}

Convexification processes are crucial in optimization. For instance, if we
consider the optimization problem%
\begin{equation*}
(\mathtt{P}_{f})\quad \limfunc{Min}f(x)\text{\quad }\mathrm{s.t.}\text{\quad 
}x\in F,
\end{equation*}%
with objective function $f:X\rightarrow \mathbb{R\cup \{+\infty \}}$, $X$
being a (real) locally convex space, and feasible set $F$, we can think on
the convenience of replacing this problem by an associated\ \emph{convex
relaxed problem}%
\begin{equation*}
(\mathtt{P}_{g})\quad \limfunc{Min}g(x)\text{\quad }\mathrm{s.t.}\text{\quad 
}x\in G,
\end{equation*}%
where $g:\mathbb{R\cup \{+\infty \}}$ is a proper convex function and $G$ is
a convex set that somehow approximate\ $f$ and $F,$ respectively. For
instance, $g$ can be the convex or closed convex hull of $f,$ and $G$ is the
convex or closed convex hull of $F$. This strategy makes sense, both
theoretically and numerically, and one is interested in conditions under
which the optimal values of both problems, $v(\mathtt{P}_{f})$ and $v(%
\mathtt{P}_{g}),$ coincide. More precisely, it is a well-known fact (e.g. 
\cite[Lemma 8.3.1]{CHL23}) that 
\begin{equation*}
v(\mathtt{P}_{f})=\limfunc{Min}_{X}(f+\mathrm{I}_{F})=\limfunc{Min}_{X}%
\overline{\limfunc{co}}(f+\mathrm{I}_{F}),
\end{equation*}%
and the desired conditions should guarantee the equality 
\begin{equation}
\overline{\limfunc{co}}(f+\mathrm{I}_{F})=(\overline{\limfunc{co}}f)+\mathrm{%
I}_{\overline{\limfunc{co}}F},  \label{ee}
\end{equation}%
which, generally, can be expressed using the biconjugate function (the
definition is given latter): 
\begin{equation*}
(f+\mathrm{I}_{F})^{\ast \ast }=f^{\ast \ast }+(\mathrm{I}_{F})^{\ast \ast }.
\end{equation*}%
Doing so, it is also interesting\ to consider this relaxation in the
possibly larger bidual space $X^{\ast \ast }.$ This problem is also
meaningful even when $f$ and $F$ are convex-lsc and closed convex,
respectively. This interest is\ justified\ by the fact that the original
problem may not have solutions, but the relaxed one can\ be solved in $%
X^{\ast \ast }$ where the bounded sets are relatively weak*-compact (thanks
to the Alaoglu-Bourbaki theorem).$\ $This process gives rise to the
so-called generalized solutions; see, for instance, \cite{EkTe76}, which
includes important examples where the closed convex hulls of $f$ and $F$
coincide with the their closed hulls, $\bar{f}$ and $\overline{F}$. As some
authors (e.g., \cite{CGH12}, \cite{FS00}, \cite{Za08}) recognize in their
works the interest of using\ the biconjugate of functions which are\textbf{\ 
}obtained by several operations which preserve convexity. The authors of 
\cite{BW08} (and references therein) also work in the latter framework, and
establish formulae for the biconjugates of some functions that appear often
in convex optimization, under suitable regularity conditions.

This article is part of the methodology described above for regularizing the
sum of functions defined in locally convex spaces. To this aim we provide\
different criteria to ensure the equalities 
\begin{equation}
\overline{\limfunc{co}}(f+g)=(\overline{\limfunc{co}}f)+(\overline{\limfunc{%
co}}g)\text{ and }(f+g)^{\ast \ast }=f^{\ast \ast }+g^{\ast \ast },
\label{eee}
\end{equation}%
by means of\ $\varepsilon $-subdifferential calculus rules for\ the sum.

In the Banach spaces\ setting, other criteria expressed in terms of exact
and fuzzy subdifferential calculus are also provided. In particular, we
extend some sequential subdifferential calculus rules for the sum of convex
functions introduced in \cite{T1, T2} for\ reflexive Banach spaces.

The structure of the paper is as follows. After a section devoted to
notation and preliminaries, Theorem \ref{prop1} in section \ref{Reg}
establishes different conditions for the fulfillment of (\ref{eee}). In
section \ref{enh} conditions are given to ensure the validity of (\ref%
{equality}) when both functions are convex but using, this time, exact
calculus rule as in \cite[Theorem 2.8.7]{Za02} and \cite{JeBu05}. Namely,
Theorem \ref{thmseq} provides\ some sequential subdifferential calculus
rules for the sum of convex functions in Banach spaces which are not
necessarily reflexive. The last section \ref{Appl} provides an illustration
of the main results of this work.

\section{Notations and preliminary results}

Let $X$ be a (real) locally convex space (lcs, in brief) and let $X^{\ast }$
be its\ topological dual space. We consider\ the bidual of $X,$ written $%
X^{\ast \ast },$ which is\ the topological dual of $X^{\ast }$ when the
latter is endowed with any locally convex topology $\tau $ such that 
\begin{equation}
\sigma (X^{\ast },X)\subset \tau \subset \beta (X^{\ast },X),  \label{tau}
\end{equation}%
where $\beta (X^{\ast },X)$ is the strong topology (for instance, the dual
norm topology in $X^{\ast }$ when $X$ is a normed space), and $\sigma
(X^{\ast },X)$ is the $w^{\ast }$-topology. Since\ $\sigma (X^{\ast
},X)\subset \tau $ (possibly, strict), we have $X=(X^{\ast },\sigma (X^{\ast
},X))^{\ast }\subset (X^{\ast },\tau )^{\ast }=X^{\ast \ast }$, and we can
identify $X$ with\ a linear subspace of $X^{\ast \ast }.$ Moreover, we
consider in $X^{\ast \ast }$ the topology $\sigma (X^{\ast \ast },X^{\ast
}), $ also denoted $w^{\ast \ast }$ for which $X$ is dense in $X^{\ast \ast
}.$ The topology $\sigma (X^{\ast \ast },X^{\ast })$ induces in $X$ the weak
topology $\sigma (X,X^{\ast }).$ In particular, if $\tau =\sigma (X^{\ast
},X),$ then $X^{\ast \ast }=X$ and $w^{\ast \ast }=\sigma (X,X^{\ast }).$ If 
$X$ is a normed space, we denote by $B_{r}(x)$ the closed ball centered at $%
x\in X$ with radius $r>0,$ and\ $B_{X}:=B_{1}(\theta ).$

Given any locally convex topology $\mathfrak{T}$ on $X$ (or $X^{\ast }$, $%
X^{\ast \ast },$ $X$), $\mathcal{N}_{X}(\mathfrak{T})$ denotes the\ basis of
neighborhoods of the origin $\theta $ that are convex, $\mathfrak{T}$-closed
and\ balanced. Given a set $A\subset X$ (or $X^{\ast },$ $X^{\ast \ast }),$
by $\limfunc{co}A$ and $\overline{\limfunc{co}}^{\mathfrak{T}}A$ we denote
the convex and the closed-convex hulls of $A$, respectively, whereas $%
\limfunc{cl}^{\mathfrak{T}}A$ and $\limfunc{int}^{\mathfrak{T}}A$ are the
respective closure and interior of $A$ (we omit the superscript $\mathfrak{T}
$ when no confusion is possible). The indicator of the set $A$ is the
function $\mathrm{I}_{A},$ which is $0$ on $A$ and $+\infty $ elsewhere.

We denote $\overline{\mathbb{R}}:=\mathbb{R\cup }\{\pm \infty \}$ and $%
\mathbb{R}_{\infty }:=\mathbb{R\cup }\{+\infty \},$ and use the convention $%
+\infty +(-\infty )=+\infty $ and $0(+\infty )=+\infty .$ Given a function $%
f:X\rightarrow \overline{\mathbb{R}},$ we define its (effective) domain $%
\limfunc{dom}f:=\{x\in X:f(x)<+\infty \}$ and epigraph $\limfunc{epi}%
f:=\{(x,\lambda )\in X\times \mathbb{R}:f(x)\leq \lambda \}.$ We say that $f$
is proper if $f>-\infty $ and $\limfunc{dom}f\neq \emptyset ,$
lower-semicontinuous (lsc, in brief) if $\limfunc{epi}f$ is closed, and
convex if $\limfunc{epi}f$ is convex.$\ $The closed and the closed convex
hulls of $f$ are the functions $\limfunc{cl}^{\mathfrak{T}}f$ and $\overline{%
\limfunc{co}}^{\mathfrak{T}}f$ whose epigraphs are $\limfunc{cl}^{\mathfrak{T%
}}(\limfunc{epi}f)$ and $\overline{\limfunc{co}}^{\mathfrak{T}}(\limfunc{epi}%
f),$ respectively.\ We denote by $\Gamma _{0}(X)$ the family of proper lsc
convex functions. The conjugate and biconjugate functions of $f$ are,
respectively, the functions $f^{\ast }:X^{\ast }\rightarrow \overline{%
\mathbb{R}}$ and $f^{\ast \ast }:X^{\ast \ast }\rightarrow \overline{\mathbb{%
R}}$ defined by 
\begin{equation*}
f^{\ast }(x^{\ast }):=\sup_{x\in X}\{\left\langle x^{\ast },x\right\rangle
-f(x)\},\text{ }x^{\ast }\in X^{\ast },
\end{equation*}%
and 
\begin{equation*}
f^{\ast \ast }(x^{\ast \ast }):=\sup_{x^{\ast }\in X^{\ast }}\{\left\langle
x^{\ast \ast },x^{\ast }\right\rangle -f^{\ast }(x^{\ast })\},\text{ }%
x^{\ast \ast }\in X^{\ast \ast }.
\end{equation*}%
Given\ $\varepsilon \in \mathbb{R}$, the $\varepsilon $-subdifferential of $%
f $ at $x\in X$ is the set\ 
\begin{equation*}
\partial _{\varepsilon }f(x):=\{x^{\ast }\in X^{\ast }:f^{\ast }(x^{\ast
})+f(x)\leq \left\langle x^{\ast },x\right\rangle +\varepsilon \},
\end{equation*}%
with the convention that $\partial _{\varepsilon }f(x)=\emptyset $ if $%
f(x)\notin \mathbb{R}$ or $\varepsilon <0.$ If $\varepsilon =0$, we recover
the usual (Fenchel) subdifferential of $f$, $\partial f:=\partial _{0}f.$ We
also use the notation $\mathrm{N}_{A}(x):=\partial \mathrm{I}_{A}(x)$ for
the normal cone to a set $A\subset X$ at $x\in X.$

The inf-convolution of\ two functions $f,$ $g:X\rightarrow \overline{\mathbb{%
R}}\ $is the function $f\square g:X\rightarrow \overline{\mathbb{R}}$
defined as 
\begin{equation*}
(f\square g)(x):=\inf \{f(x_{1})+g(x_{2}):x_{1}+x_{2}=x\}.
\end{equation*}%
We always have that\ 
\begin{equation*}
\left( f\square g\right) ^{\ast }=f^{\ast }+g^{\ast }.
\end{equation*}%
Additionally, if $f,$ $g\in \Gamma _{0}(X)$ are such that $\limfunc{dom}%
f\cap \limfunc{dom}g\neq \emptyset ,$ then 
\begin{equation}
(f+g)^{\ast }=\overline{f^{\ast }\square g^{\ast }}^{\sigma (X^{\ast },X)}.
\label{clinf}
\end{equation}%
This fact comes from the Fenchel-Moreau-Rockafellar theorem which states
that, for every function $f:X\rightarrow \mathbb{R}_{\infty },$ we have 
\begin{equation*}
(f^{\ast \ast })_{\mid X}=\overline{\limfunc{co}}f,
\end{equation*}%
provided that $\overline{\limfunc{co}}f$ is proper and where $(f^{\ast \ast
})_{\mid X}$ is\ the restriction of $f^{\ast \ast }$ to $X.$

For any\ real extended-valued function $f:X\rightarrow \overline{\mathbb{R}}%
, $ we define\ its extension to $X^{\ast \ast }$ by 
\begin{equation*}
\hat{f}(x):=\left\{ 
\begin{array}{ll}
f(x), & \text{if }x\in X, \\ 
+\infty , & \text{if }x\in X^{\ast \ast }\setminus X.%
\end{array}%
\right.
\end{equation*}%
For simplicity, the extension $\hat{f}$ is usually denoted as the original
function $f$ when no confusion occurs. It is easy to see that $(\hat{f}%
)^{\ast }=f^{\ast }$ and, for every $x\in X$ and $\varepsilon \in \mathbb{R}$%
, 
\begin{equation*}
\partial _{\varepsilon }\hat{f}(x)=\partial _{\varepsilon }f(x).
\end{equation*}%
By $\overline{f}^{w^{\ast \ast }}:X^{\ast \ast }\rightarrow \overline{%
\mathbb{R}}$ we denote the $w^{\ast \ast }$-closed hull of (the extension
of) $f:X\rightarrow \overline{\mathbb{R}};$ equivalently, 
\begin{equation}
\overline{f}^{w^{\ast \ast }}(x)=\liminf_{y\overset{w^{\ast \ast }}{%
\rightarrow }x,\text{ }y\in X}f(y):=\sup_{U\in \mathcal{N}}\inf_{y\in
x+U}f(y),\text{ }x\in X^{\ast \ast },  \label{marco1}
\end{equation}%
where $\mathcal{N}:=\mathcal{N}_{X^{\ast \ast }}(w^{\ast \ast })$ and $y%
\overset{w^{\ast \ast }}{\rightarrow }x$ (usually also written $%
y\rightharpoonup x$) refers to the convergence in $(X^{\ast \ast },w^{\ast
\ast }).$ Moreover,\ the $w^{\ast \ast }$-closed convex hull of (the
extension of) $f$ is the function $\overline{\limfunc{co}}^{w^{\ast \ast
}}f:=\overline{\limfunc{co}}^{w^{\ast \ast }}(\hat{f}).$ The functions $%
\overline{\limfunc{co}}^{w^{\ast \ast }}f$ and $\overline{f}^{w^{\ast \ast
}}\ $are the supremum of all the lsc convex and all the lsc minorants of $f,$
respectively.\ It is also worth observing that\ 
\begin{equation*}
(\overline{\limfunc{co}}^{w^{\ast \ast }}f)^{\ast }=(\overline{f}^{w^{\ast
\ast }})^{\ast }=f^{\ast }.
\end{equation*}%
When $f\in \Gamma _{0}(X)$ and $\tau =\sigma (X^{\ast },X),$ we have $%
\overline{\limfunc{co}}^{w^{\ast \ast }}f=\overline{f}^{w^{\ast \ast }}=%
\overline{f}^{\sigma (X,X^{\ast })}=f,$ due to Mazur's theorem. If\ $f$ is
proper,$\ $the function $\overline{\limfunc{co}}^{w^{\ast \ast }}f$ is
proper if and only if $f$ has a continuous\ affine minorant. Indeed, $%
\overline{\limfunc{co}}^{w^{\ast \ast }}f$ is proper if and only if it has a
continuous\ affine minorant (see, e.g., \cite[Proposition 3.1.4]{CHL23}), so
if and only if $f$ also has.

\section{Regularization of the sum\label{Reg}}

In this section, given functions $f$ and $g$ defined on the lcs $X,$ we
establish different conditions to guarantee the equality\ 
\begin{equation}
\overline{\limfunc{co}}^{w^{\ast \ast }}(f+g)=(\overline{\limfunc{co}}%
^{w^{\ast \ast }}f)+(\overline{\limfunc{co}}^{w^{\ast \ast }}g).
\label{equality}
\end{equation}%
These conditions rely on the validity of\ appropriate $\varepsilon $%
-subdifferential calculus rules for\ the sum $f+g.$ A first condition\ has
been introduced in \cite[Proposition 5.5]{CGH12}, and uses the following
calculus rule, for $x\in \limfunc{dom}f\cap \limfunc{dom}g$ and $\varepsilon
>0,$ 
\begin{equation}
\partial _{\varepsilon }(f+g)(x)=\tbigcap_{\alpha >\varepsilon }\limfunc{cl}%
\nolimits^{\tau }\left( \tbigcup_{\substack{ \varepsilon _{1}+\varepsilon
_{2}=\alpha  \\ \varepsilon _{1},\varepsilon _{2}\geq 0}}\partial
_{\varepsilon _{1}}f(x)+\partial _{\varepsilon _{2}}g(x)\right) .
\label{sum1b}
\end{equation}%
Observe that the set between parenthesis is always convex and, so, instead
of $\tau $ one can use any other compatible topology for the pair $(X^{\ast
},X^{\ast \ast }),$ in particular the topology $\sigma (X^{\ast },X^{\ast
\ast })$.

In the present work, we provide an alternative condition by means of the
sharper formula,%
\begin{equation}
\partial _{\varepsilon }(f+g)(x)=\limfunc{cl}\nolimits^{\tau }\left(
\tbigcup _{\substack{ \varepsilon _{1}+\varepsilon _{2}=\varepsilon  \\ %
\varepsilon _{1},\varepsilon _{2}\geq 0}}\partial _{\varepsilon
_{1}}f(x)+\partial _{\varepsilon _{2}}g(x)\right) .  \label{summ1}
\end{equation}%
Notice that\ (\ref{summ1}) reduces\ to the well-known Hiriart-Urruty \&
Phelps formula in the framework of convex analysis (\cite{HMSV95}; see,
also, \cite{CHL23}, \cite{CHL16}, and \cite{HLZ08}), provided that\ $f,$ $%
g\in \Gamma _{0}(X)$ and $\tau =\sigma (X^{\ast },X)$ (in this case $X^{\ast
\ast }=X$). Our characterizations will also involve the following inclusion,
for $x\in X$ and $\varepsilon >0,$ 
\begin{equation}
\partial _{\varepsilon }(f+g)(x)\subset \limfunc{cl}\nolimits^{\tau
}(\partial _{2\varepsilon }f(x)+\partial _{2\varepsilon }g(x)).
\label{sum1d}
\end{equation}

We first state a useful equivalence of (\ref{equality}) with a conjugation
sum rule. This result, which is used later in the proof of\ Theorem \ref%
{prop1}, is a nonconvex counterpart\ of the corresponding property in \cite[%
Theorem 2.5]{BG08}.

\begin{prop}
\label{lema1} Let $f,$ $g:X\rightarrow \mathbb{R}_{\infty }$ be functions
such that $\limfunc{dom}f\cap \limfunc{dom}g\neq \emptyset $ and each one
has\ a continuous\ affine minorant. Then (\ref{equality}) is equivalent to 
\begin{equation}
(f+g)^{\ast }=\overline{f^{\ast }\square g^{\ast }}^{\tau }.  \label{no}
\end{equation}
\end{prop}

\begin{dem}
If\ (\ref{equality})$\ $holds, then (\ref{clinf}) entails 
\begin{eqnarray*}
(f+g)^{\ast } &=&(\overline{\limfunc{co}}^{w^{\ast \ast }}(f+g))^{\ast }=((%
\overline{\limfunc{co}}^{w^{\ast \ast }}f)+(\overline{\limfunc{co}}^{w^{\ast
\ast }}g))^{\ast } \\
&=&\overline{(\overline{\limfunc{co}}^{w^{\ast \ast }}f)^{\ast }\square (%
\overline{\limfunc{co}}^{w^{\ast \ast }}g)^{\ast }}^{\tau }=\overline{%
f^{\ast }\square g^{\ast }}^{\tau },
\end{eqnarray*}%
and (\ref{no}) follows. Conversely, observe that 
\begin{equation*}
\overline{\limfunc{co}}^{w^{\ast \ast }}(f+g)\geq \overline{\limfunc{co}}%
^{w^{\ast \ast }}(\overline{\limfunc{co}}^{w^{\ast \ast }}f+\overline{%
\limfunc{co}}^{w^{\ast \ast }}g)=\left( \overline{\limfunc{co}}^{w^{\ast
\ast }}f\right) +\left( \overline{\limfunc{co}}^{w^{\ast \ast }}g\right)
>-\infty ,
\end{equation*}%
and $\overline{\limfunc{co}}^{w^{\ast \ast }}(f+g)\leq f+g$, so $\overline{%
\limfunc{co}}^{w^{\ast \ast }}(f+g)$ is proper. Then, taking the conjugate
in\ (\ref{no}) and\ applying twice Moreau's theorem in $(X^{\ast },X^{\ast
\ast }),$ we obtain\ 
\begin{eqnarray*}
\overline{\limfunc{co}}^{w^{\ast \ast }}(f+g) &=&(f+g)^{\ast \ast }=\left( 
\overline{f^{\ast }\square g^{\ast }}^{\tau }\right) ^{\ast } \\
&=&\left( f^{\ast }\square g^{\ast }\right) ^{\ast }=f^{\ast \ast }+g^{\ast
\ast }=(\overline{\limfunc{co}}^{w^{\ast \ast }}f)+(\overline{\limfunc{co}}%
^{w^{\ast \ast }}g).
\end{eqnarray*}
\end{dem}

Next, we introduce the following quantity (possibly, $+\infty $) associated
with a function $f:X\rightarrow \overline{\mathbb{R}}$ and $x\in \limfunc{dom%
}f,$ 
\begin{equation}
\varepsilon _{f}(x):=\inf \{\alpha >0:\partial _{\alpha }f(x)\neq \emptyset
\}.  \label{epsilonx}
\end{equation}%
The following lemma provides some information on this quantity.

\begin{lem}
\label{lema2}The following assertions hold, for all function $f:X\rightarrow 
\overline{\mathbb{R}}$ and all $x\in \limfunc{dom}f.$

$(i)$ If $f\in \Gamma _{0}(X),$ then $\varepsilon _{f}(x)=0.$

$(ii)$\ If $f^{\ast }$\ is proper, then $\varepsilon _{f}(x)<+\infty .$

$(iii)$ If $\overline{\limfunc{co}}^{w^{\ast \ast }}f$ is\ proper, then 
\begin{equation*}
\inf_{z\in \limfunc{dom}f}\varepsilon _{f}(z)=0.
\end{equation*}
\end{lem}

\begin{dem}
$(i)$ If $f\in \Gamma _{0}(X),$ then $\partial _{\alpha }f(x)\neq \emptyset $
for all $\alpha >0.$

$(ii)$ If $f^{\ast }$\ is proper, then $f$ is also proper. Moreover, for
every $x\in \limfunc{dom}f$ and $x^{\ast }\in \limfunc{dom}f^{\ast }$, we
have $f^{\ast }(x^{\ast })+f(x)-\left\langle x^{\ast },x\right\rangle \geq
0. $ So, every scalar $\alpha $\ such that $\alpha >f^{\ast }(x^{\ast
})+f(x)-\left\langle x^{\ast },x\right\rangle $ satisfies $x^{\ast }\in
\partial _{\alpha }f(x),$ and we deduce $\varepsilon _{f}(x)\leq \alpha .$

$(iii)$ Since $\overline{\limfunc{co}}^{w^{\ast \ast }}f$ is proper, its
conjugate $f^{\ast }=(\overline{\limfunc{co}}^{w^{\ast \ast }}f)^{\ast }$ is
also proper\ (e.g., \cite[Proposition 3.1.4]{CHL23}). Take $x^{\ast }\in 
\limfunc{dom}f^{\ast }$ and pick any $\rho >0.$ Then there exists some\ $%
x_{\rho }\in \limfunc{dom}f$ such that 
\begin{equation*}
f^{\ast }(x^{\ast })+f(x_{\rho })-\left\langle x^{\ast },x_{\rho
}\right\rangle \leq \rho ;
\end{equation*}%
that is, $x^{\ast }\in \partial _{\rho }f(x_{\rho })$ and 
\begin{equation*}
\inf_{z\in \limfunc{dom}f}\varepsilon _{f}(z)\leq \varepsilon _{f}(x_{\rho
})\leq \rho .
\end{equation*}%
Hence, $\inf_{z\in \limfunc{dom}f}\varepsilon _{f}(z)=0$ by\ the
arbitrariness of $\rho >0.$
\end{dem}

The following result provides the announced characterizations of (\ref%
{equality}).

\begin{theo}
\label{prop1}For every functions $f$ and $g$ as in Lemma \ref{lema1}, the
following statements are equivalent:$\medskip $

$(i)$ $\overline{\limfunc{co}}^{w^{\ast \ast }}(f+g)=(\overline{\limfunc{co}}%
^{w^{\ast \ast }}f)+(\overline{\limfunc{co}}^{w^{\ast \ast }}g).\medskip $

$(ii)$ Formula (\ref{summ1}) holds,\ for all $x\in \limfunc{dom}f\cap 
\limfunc{dom}g$ and $\varepsilon >\varepsilon _{f+g}(x).\medskip $

$(iii)$ Formula (\ref{sum1b}) holds,\ for all $x\in \limfunc{dom}f\cap 
\limfunc{dom}g$ and $\varepsilon >0.\medskip $

$(iv)$\ The inclusion\ (\ref{sum1d}) holds,\ for all $x\in X$ and $%
\varepsilon >0.$
\end{theo}

\begin{dem}
The equivalence between $(i)$ and $(iii)$ is established in \cite[%
Proposition 5.5]{CGH12}, and we prove next\ that $(i)\Longleftrightarrow
(ii) $,\ $(iv)\implies (i)$ and $(iii)\implies (iv).$

$(i)\implies (ii).$\ Assume that $(i)$ holds or, equivalently according to
Lemma \ref{lema1}, 
\begin{equation}
(f+g)^{\ast }=\overline{f^{\ast }\square g^{\ast }}^{\tau }.  \label{es}
\end{equation}%
Fix $x\in \limfunc{dom}f\cap \limfunc{dom}g$ and $\varepsilon >\varepsilon
_{f+g}(x)$ (remember that $\varepsilon _{f+g}(x)<\infty $ by Lemma \ref%
{lema2}$(ii)$). Take $x^{\ast }\in \partial _{\varepsilon }(f+g)(x)$ and a $%
\theta $-neighborhood $U$ in $(X^{\ast },\tau ).$ Since $\varepsilon
>\varepsilon _{f+g}(x)$, there exists some $0<\alpha <\varepsilon $ such
that $\partial _{\alpha }(f+g)(x)\neq \emptyset ;$ that is, $\cup _{0<\delta
<\varepsilon }\partial _{\varepsilon -\delta }(f+g)(x)\neq \emptyset .$ So,
by applying \cite[Proposition 4.1.9]{CHL23} to the function $\hat{f}+\hat{g}$
defined on $X^{\ast \ast }$, we obtain 
\begin{eqnarray*}
x^{\ast } &\in &\partial _{\varepsilon }(f+g)(x)=\partial _{\varepsilon }(%
\hat{f}+\hat{g})(x) \\
&=&\limfunc{cl}\nolimits^{\tau }\left( \tbigcup_{0<\delta <\varepsilon
}\partial _{\varepsilon -\delta }(\hat{f}+\hat{g})(x)\right) =\limfunc{cl}%
\nolimits^{\tau }\left( \tbigcup_{0<\delta <\varepsilon }\partial
_{\varepsilon -\delta }(f+g)(x)\right) .
\end{eqnarray*}%
Hence, we find $\delta \in (0,\varepsilon )$ and $y^{\ast }\in \partial
_{\varepsilon -\delta }(f+g)(x)$ such that $x^{\ast }-y^{\ast }\in U.$
Moreover, since (remember (\ref{es})) 
\begin{eqnarray*}
(f+g)(x)+\overline{f^{\ast }\square g^{\ast }}^{\tau }(y^{\ast })
&=&(f+g)(x)+(f+g)^{\ast }(y^{\ast }) \\
&\leq &\left\langle y^{\ast },x\right\rangle +\varepsilon -\delta ,
\end{eqnarray*}%
we find some $z^{\ast }\in y^{\ast }+U$ such that 
\begin{equation*}
(f+g)(x)+\left( f^{\ast }\square g^{\ast }\right) (z^{\ast })\leq
\left\langle z^{\ast },x\right\rangle +\varepsilon -\delta /2<\left\langle
z^{\ast },x\right\rangle +\varepsilon .
\end{equation*}%
In other words, there are\ $z_{1}^{\ast },$ $z_{2}^{\ast }\in X^{\ast }$
such that $z_{1}^{\ast }+z_{2}^{\ast }=z^{\ast }$ and 
\begin{equation*}
(f+g)(x)+f^{\ast }(z_{1}^{\ast })+g^{\ast }(z_{2}^{\ast })\leq \left\langle
z_{1}^{\ast }+z_{2}^{\ast },x\right\rangle +\varepsilon ;
\end{equation*}%
that is, rearranging these terms,%
\begin{equation*}
(f(x)+f^{\ast }(z_{1}^{\ast })-\left\langle z_{1}^{\ast },x\right\rangle
)+(g(x)+g^{\ast }(z_{2}^{\ast })-\left\langle z_{2}^{\ast },x\right\rangle
)\leq \varepsilon .
\end{equation*}%
Now, we denote\ 
\begin{equation*}
\varepsilon _{1}:=f(x)+f^{\ast }(z_{1}^{\ast })-\left\langle z_{1}^{\ast
},x\right\rangle \text{ and }\varepsilon _{2}:=\varepsilon -\varepsilon _{1},
\end{equation*}%
so that $\varepsilon _{1}\geq 0$ and $\varepsilon _{2}\geq g(x)+g^{\ast
}(z_{2}^{\ast })-\left\langle z_{2}^{\ast },x\right\rangle \geq 0,$ and we
deduce that $\varepsilon _{1}+\varepsilon _{2}=\varepsilon ,$ $z_{1}^{\ast
}\in \partial _{\varepsilon _{1}}f(x)$ and $z_{2}^{\ast }\in \partial
_{\varepsilon _{2}}g(x).$ Thus,\ 
\begin{equation*}
x^{\ast }\in y^{\ast }+U\subset z^{\ast }+2U=z_{1}^{\ast }+z_{2}^{\ast
}+2U\subset \partial _{\varepsilon _{1}}f(x)+\partial _{\varepsilon
_{2}}g(x)+2U,
\end{equation*}%
and\ $(iii)$ follows by the arbitrariness of $U$.

$(ii)\implies (i).$ As observed in the proof of Lemma \ref{lema2}, remember\
that the function $\overline{\limfunc{co}}^{w^{\ast \ast }}(f+g)$ is proper
and, so, its conjugate $(f+g)^{\ast }$ is also proper. First, we are going
to prove that, for every given $x^{\ast }\in \limfunc{dom}(f+g)^{\ast },$ 
\begin{equation}
((\overline{\limfunc{co}}^{w^{\ast \ast }}f)+(\overline{\limfunc{co}}%
^{w^{\ast \ast }}g))^{\ast }(x^{\ast })\leq (f+g)^{\ast }(x^{\ast }).
\label{des1}
\end{equation}%
Fix $\varepsilon >0$ and choose $x_{\varepsilon }\in X$ such that 
\begin{equation}
(f+g)^{\ast }(x^{\ast })+f(x_{\varepsilon })+g(x_{\varepsilon
})-\left\langle x^{\ast },x_{\varepsilon }\right\rangle <\varepsilon ;
\label{in1}
\end{equation}%
that is, $x^{\ast }\in \partial _{\varepsilon }(f+g)(x_{\varepsilon })$ and $%
\varepsilon _{f+g}(x_{\varepsilon })<\varepsilon .$ So, by $(ii)$ we have\ 
\begin{equation}
x^{\ast }\in \limfunc{cl}\nolimits^{\tau }\left( \tbigcup_{\substack{ %
\varepsilon _{1}+\varepsilon _{2}=\varepsilon  \\ \varepsilon
_{1},\varepsilon _{2}\geq 0}}\partial _{\varepsilon _{1}}f(x_{\varepsilon
})+\partial _{\varepsilon _{2}}g(x_{\varepsilon })\right) ,  \label{dos}
\end{equation}%
which reads, due to the inclusions $\partial _{\varepsilon
_{1}}f(x_{\varepsilon })\subset \partial _{\varepsilon _{1}}(\overline{%
\limfunc{co}}^{w^{\ast \ast }}f)(x_{\varepsilon })$ and $\partial
_{\varepsilon _{2}}g(x_{\varepsilon })\subset \partial _{\varepsilon _{2}}(%
\overline{\limfunc{co}}^{w^{\ast \ast }}g)(x_{\varepsilon }),$%
\begin{eqnarray}
x^{\ast } &\in &\limfunc{cl}\nolimits^{\tau }\left( \tbigcup_{\substack{ %
\varepsilon _{1}+\varepsilon _{2}=\varepsilon  \\ \varepsilon
_{1},\varepsilon _{2}\geq 0}}\partial _{\varepsilon _{1}}(\overline{\limfunc{%
co}}^{w^{\ast \ast }}f)(x_{\varepsilon })+\partial _{\varepsilon _{2}}(%
\overline{\limfunc{co}}^{w^{\ast \ast }}g)(x_{\varepsilon })\right)  \notag
\\
&\subset &\limfunc{cl}\nolimits^{\tau }(\partial _{\varepsilon }((\overline{%
\limfunc{co}}^{w^{\ast \ast }}f)+(\overline{\limfunc{co}}^{w^{\ast \ast
}}g))(x_{\varepsilon }))=\partial _{\varepsilon }((\overline{\limfunc{co}}%
^{w^{\ast \ast }}f)+(\overline{\limfunc{co}}^{w^{\ast \ast
}}g))(x_{\varepsilon });  \notag
\end{eqnarray}%
that is, 
\begin{equation}
((\overline{\limfunc{co}}^{w^{\ast \ast }}f)+(\overline{\limfunc{co}}%
^{w^{\ast \ast }}g))^{\ast }(x^{\ast })+(\overline{\limfunc{co}}^{w^{\ast
\ast }}f)(x_{\varepsilon })+(\overline{\limfunc{co}}^{w^{\ast \ast
}}g)(x_{\varepsilon })-\left\langle x^{\ast },x_{\varepsilon }\right\rangle
\leq \varepsilon .  \label{inv}
\end{equation}%
At the same time, due to (\ref{dos}), the sets $\partial _{\varepsilon
_{1}}f(x_{\varepsilon })$ and $\partial _{\varepsilon _{2}}g(x_{\varepsilon
})$ are nonempty for some $\varepsilon _{1},\varepsilon _{2}\geq 0$ such
that $\varepsilon _{1}+\varepsilon _{2}=\varepsilon ,$ thus\ (e.g., \cite[%
Exercise 62]{CHL23}) 
\begin{equation*}
(\overline{\limfunc{co}}^{w^{\ast \ast }}f)(x_{\varepsilon })\geq
f(x_{\varepsilon })-\varepsilon _{1},\text{ }(\overline{\limfunc{co}}%
^{w^{\ast \ast }}g)(x_{\varepsilon })\geq g(x_{\varepsilon })-\varepsilon
_{2},
\end{equation*}%
and\ (\ref{in1}) and (\ref{inv}) entail%
\begin{multline*}
((\overline{\limfunc{co}}^{w^{\ast \ast }}f)+(\overline{\limfunc{co}}%
^{w^{\ast \ast }}g))^{\ast }(x^{\ast })-(f+g)^{\ast }(x^{\ast }) \\
\leq ((\overline{\limfunc{co}}^{w^{\ast \ast }}f)+(\overline{\limfunc{co}}%
^{w^{\ast \ast }}g))^{\ast }(x^{\ast })+f(x_{\varepsilon })+g(x_{\varepsilon
})-\left\langle x^{\ast },x_{\varepsilon }\right\rangle \\
\leq ((\overline{\limfunc{co}}^{w^{\ast \ast }}f)+(\overline{\limfunc{co}}%
^{w^{\ast \ast }}g))^{\ast }(x^{\ast })+(\overline{\limfunc{co}}^{w^{\ast
\ast }}f)(x_{\varepsilon })+(\overline{\limfunc{co}}^{w^{\ast \ast
}}g)(x_{\varepsilon })-\left\langle x^{\ast },x_{\varepsilon }\right\rangle
+\varepsilon \\
\leq 2\varepsilon .
\end{multline*}%
Hence, (\ref{des1}) follows by the arbitrariness of $\varepsilon $.
Moreover, since\ the\ inequality $((\overline{\limfunc{co}}^{w^{\ast \ast
}}f)+(\overline{\limfunc{co}}^{w^{\ast \ast }}g))^{\ast }\geq (f+g)^{\ast }$
always holds, we deduce that $((\overline{\limfunc{co}}^{w^{\ast \ast }}f)+(%
\overline{\limfunc{co}}^{w^{\ast \ast }}g))^{\ast }=(f+g)^{\ast }$.
Therefore $(i)$ follows by applying Moreau's theorem to the last equality in
the pair $(X^{\ast },X^{\ast \ast }).$

$(iii)\implies (iv).$ For any\ $x\in X$ and $\varepsilon >0,$ $(iii)$
implies that 
\begin{equation}
\partial _{\varepsilon }(f+g)(x)\subset \limfunc{cl}\nolimits^{\tau
}(\partial _{2\varepsilon }f(x)+\partial _{2\varepsilon }g(x)),
\label{est0v}
\end{equation}%
which is $(iv).$

$(iv)\implies (i).$ Fix $x\in X$ and $\varepsilon >0.$ Since $\partial
_{2\varepsilon }f(x)\subset \partial _{2\varepsilon }(\overline{\limfunc{co}}%
^{w^{\ast \ast }}f)(x)$ and $\partial _{2\varepsilon }g(x)\subset \partial
_{2\varepsilon }(\overline{\limfunc{co}}^{w^{\ast \ast }}g)(x),$ $(iv)$
entails 
\begin{eqnarray*}
\partial _{\varepsilon }(f+g)(x) &\subset &\limfunc{cl}\nolimits^{\tau
}(\partial _{2\varepsilon }(\overline{\limfunc{co}}^{w^{\ast \ast
}}f)(x)+\partial _{2\varepsilon }(\overline{\limfunc{co}}^{w^{\ast \ast
}}g)(x)) \\
&\subset &\limfunc{cl}\nolimits^{\tau }(\partial _{4\varepsilon }((\overline{%
\limfunc{co}}^{w^{\ast \ast }}f)+(\overline{\limfunc{co}}^{w^{\ast \ast
}}g))(x)) \\
&=&\partial _{4\varepsilon }((\overline{\limfunc{co}}^{w^{\ast \ast }}f)+(%
\overline{\limfunc{co}}^{w^{\ast \ast }}g))(x).
\end{eqnarray*}%
Now, since $(f+g)^{\ast }=(\overline{\limfunc{co}}^{w^{\ast \ast
}}(f+g))^{\ast }$ is proper, by applying \cite[Corollary 3.3 (see, also,
Remark 3.2(i))]{CGH18} in $X^{\ast \ast }$ (see, also, \cite{CHP17, CHS16,
LoVo14}) we deduce that 
\begin{equation}
\overline{\limfunc{co}}^{w^{\ast \ast }}(f+g)=\overline{\limfunc{co}}%
^{w^{\ast \ast }}(\hat{f}+\hat{g})=(\overline{\limfunc{co}}^{w^{\ast \ast
}}f)+(\overline{\limfunc{co}}^{w^{\ast \ast }}g)+c,  \label{est}
\end{equation}%
for some constant $c\in \mathbb{R}$. It is clear that $c\geq 0$ and let us
check that $c=0.$ Since $\overline{\limfunc{co}}^{w^{\ast \ast }}(f+g)$ is
proper, according to Lemma \ref{lema2}$(iii),$ for every $\varepsilon >0$
there exists some $x_{\varepsilon }\in X$ such that $\varepsilon
_{f+g}(x_{\varepsilon })<\varepsilon ;$ that is, $\partial _{\varepsilon
}(f+g)(x_{\varepsilon })\neq \emptyset .$ Thus, by (\ref{est0v}), the sets $%
\partial _{2\varepsilon }f(x_{\varepsilon })$ and $\partial _{2\varepsilon
}g(x)$ are nonempty, and so\ 
\begin{equation*}
(\overline{\limfunc{co}}^{w^{\ast \ast }}f)(x_{\varepsilon })\geq
f(x_{\varepsilon })-2\varepsilon ,\text{ }(\overline{\limfunc{co}}^{w^{\ast
\ast }}g)(x_{\varepsilon })\geq g(x_{\varepsilon })-2\varepsilon .
\end{equation*}%
Thus,\ (\ref{est}) reads\ 
\begin{eqnarray*}
c &=&\overline{\limfunc{co}}^{w^{\ast \ast }}(f+g)(x_{\varepsilon })-(%
\overline{\limfunc{co}}^{w^{\ast \ast }}f)(x_{\varepsilon })-(\overline{%
\limfunc{co}}^{w^{\ast \ast }}g)(x_{\varepsilon }) \\
&\leq &\overline{\limfunc{co}}^{w^{\ast \ast }}(f+g)(x_{\varepsilon
})-f(x_{\varepsilon })-g(x_{\varepsilon })+4\varepsilon \leq 4\varepsilon ,
\end{eqnarray*}%
and taking $\varepsilon \downarrow 0$ we conclude that\ $c=0.$
\end{dem}

\section{Enhanced criteria in the convex framework\label{enh}}

In this section we explore other criteria for the validity of\ 
\begin{equation}
\overline{f+g}^{w^{\ast \ast }}=\overline{f}^{w^{\ast \ast }}+\overline{g}%
^{w^{\ast \ast }},  \label{equalityco}
\end{equation}%
for convex functions $f$ and $g$ defined on the lcs $X.$ Compared with the
previous section, the conditions now are expressed in terms of exact and
fuzzy subdifferential calculus rules for the sum $f+g.$

We will apply the following lemma, which is also of independent interest.
Actually, it is a compact form of \cite[Proposition 1]{Ro70} (see, also, 
\cite[Lemma 2.3]{CGH18} and \cite[Theorem 2.2 and Remark 2.4]{Vo08} for
other variants in locally convex spaces). A proof is given for completeness.

\begin{lem}
\label{lembb} Assume that\ $X$ is Banach. If $f\in \Gamma _{0}(X),$ then\
for all $z\in X^{\ast \ast }$\ we have that 
\begin{equation}
\partial \bar{f}^{w^{\ast \ast }}(z)=\tbigcap_{U\in \mathcal{N}}~\limfunc{cl}%
\nolimits^{\left\Vert \cdot \right\Vert _{\ast }}\left( \tbigcup_{x\in
z+U}\partial f(x)\right) ,  \label{mm}
\end{equation}%
where $\left\Vert \cdot \right\Vert _{\ast }$ is the dual norm in $X^{\ast }$
and $\mathcal{N}:=\mathcal{N}_{X^{\ast \ast }}(w^{\ast \ast }).$
\end{lem}

\begin{proof}
Pick\ $z\in X^{\ast \ast },$ $x^{\ast }\in \partial \bar{f}^{w^{\ast \ast
}}(z),$ $U\in z+\mathcal{N}$ , and choose $\varepsilon >0$ and $V\in z+%
\mathcal{N}$ such that $V+\sqrt{\varepsilon }B_{X}\subset U.$\ Then 
\begin{equation*}
\bar{f}^{w^{\ast \ast }}(z)+f^{\ast }(x^{\ast })=\bar{f}^{w^{\ast \ast
}}(z)+(\bar{f}^{w^{\ast \ast }})^{\ast }(x^{\ast })\leq \left\langle
z,x^{\ast }\right\rangle <\left\langle z,x^{\ast }\right\rangle +\varepsilon
,
\end{equation*}%
and, so, there exists some $x_{\varepsilon }\in V$ such that\ 
\begin{equation*}
f(x_{\varepsilon })+f^{\ast }(x^{\ast })<\left\langle x_{\varepsilon
},x^{\ast }\right\rangle +\varepsilon .
\end{equation*}%
Hence, $x^{\ast }\in \partial _{\varepsilon }f(x_{\varepsilon })$ and the
Brondsted-Rockafellar theorem (see, e.g., \cite[Proposition 4.3.7]{CHL23})
gives some $z_{\varepsilon }\in x_{\varepsilon }+\sqrt{\varepsilon }B_{X}$\
and $z_{\varepsilon }^{\ast }\in x^{\ast }+\sqrt{\varepsilon }B_{X^{\ast }}$
such $z_{\varepsilon }^{\ast }\in \partial f(z_{\varepsilon }).$ Thus, $%
z_{\varepsilon }\in V+\sqrt{\varepsilon }B_{X}\subset U$ and 
\begin{equation*}
x^{\ast }\in z_{\varepsilon }^{\ast }+\sqrt{\varepsilon }B_{X^{\ast
}}\subset \partial f(z_{\varepsilon })+\sqrt{\varepsilon }B_{X^{\ast
}}\subset \tbigcup_{x\in U}\partial f(x)+\sqrt{\varepsilon }B_{X^{\ast }}.
\end{equation*}%
By the arbitrariness of $\varepsilon $ we deduce that $x^{\ast }\in \limfunc{%
cl}^{\left\Vert \cdot \right\Vert _{\ast }}\left( \cup _{x\in U}\partial
f(x)\right) $ and the arbitrariness of $U$ yields the inclusion
\textquotedblleft $\subset $\textquotedblright\ in (\ref{mm}).

Conversely, if $x^{\ast }$ is in\ the right hand side of (\ref{mm}),$\ $then
for each $U\in \mathcal{N}$ and $\varepsilon >0$ there exist some $x\in z+U$
and $z^{\ast }\in \partial f(x)$ such that $\left\Vert x^{\ast }-z^{\ast
}\right\Vert _{\ast }\leq \varepsilon $ and $f^{\ast }(x^{\ast })\leq
f^{\ast }(z^{\ast })+\varepsilon ,$ due to the $\left\Vert \cdot \right\Vert
_{\ast }$-lower semicontinuity of $f^{\ast }.\ $Hence, 
\begin{eqnarray*}
\inf_{y\in U}f(z+y)+f^{\ast }(x^{\ast }) &\leq &f(x)+f^{\ast }(x^{\ast }) \\
&\leq &f(x)+f^{\ast }(z^{\ast })+\varepsilon \\
&=&\left\langle x,z^{\ast }\right\rangle +\varepsilon \leq \left\langle
x,x^{\ast }\right\rangle +\left\Vert x\right\Vert \varepsilon +\varepsilon ,
\end{eqnarray*}%
and, by taking the supremum over $U$, 
\begin{equation*}
\bar{f}^{w^{\ast \ast }}(z)+(\bar{f}^{w^{\ast \ast }})^{\ast }(x^{\ast
})=\sup_{U}\inf_{y\in z+U}f(y)+f^{\ast }(x^{\ast })\leq \left\langle
x,x^{\ast }\right\rangle +\left\Vert x\right\Vert \varepsilon +\varepsilon .
\end{equation*}%
Thus, as $\varepsilon \downarrow 0$ we obtain that\ $x^{\ast }\in \partial
f^{w^{\ast \ast }}(z).$
\end{proof}

Conditions $(vi)$ in the following theorem\ completes the previous set of
conditions $(ii)$-$(iv)$ in Theorem \ref{prop1}.

\begin{theo}
\label{thmseq}Assume that $X$ is Banach, and let $f,$ $g\in \Gamma _{0}(X)$
such that $\limfunc{dom}f\cap \limfunc{dom}g\neq \emptyset .$ Then the
following statements are equivalent, provided that $\tau $ is the dual norm\
topology (see (\ref{tau})):

$(v)$ $\overline{f+g}^{w^{\ast \ast }}=\overline{f}^{w^{\ast \ast }}+%
\overline{g}^{w^{\ast \ast }}.$

$(vi)$ For all $x\in X$ and $x^{\ast }\in \partial (f+g)(x)$, there are
sequences $x_{n},$ $y_{n}\rightarrow x$\ and $x_{n}^{\ast }\in \partial
f(x_{n}),$ $y_{n}^{\ast }\in \partial g(y_{n})$ such that 
\begin{equation*}
\left\langle x_{n}^{\ast },x_{n}-x\right\rangle \rightarrow 0,\text{ }%
\left\langle y_{n}^{\ast },y_{n}-x\right\rangle \rightarrow 0,\text{ }%
f(x_{n})\rightarrow f(x),\text{ }g(y_{n})\rightarrow g(x),
\end{equation*}%
and 
\begin{equation*}
x_{n}^{\ast }+y_{n}^{\ast }\overset{\left\Vert \cdot \right\Vert _{\ast }}{%
\rightarrow }x^{\ast }.
\end{equation*}
\end{theo}

\begin{dem}
$(vi)\implies (v):$ Given\ $x\in X$ and $x^{\ast }\in \partial (f+g)(x),$ we
let the sequences $(x_{n}),$ $(y_{n})\subset X$, $x_{n}^{\ast }\in \partial
f(x_{n})$ and $y_{n}^{\ast }\in \partial g(y_{n})$ be as\ in $(vi).$\ Then,
given any $\varepsilon >0,$ for all $n$ large enough\ we have 
\begin{equation*}
x^{\ast }\in x_{n}^{\ast }+y_{n}^{\ast }+\varepsilon B_{X^{\ast }},\text{ }%
\left\langle x_{n}^{\ast },x_{n}-x\right\rangle \leq \varepsilon /2,\text{ }%
f(x)-f(x_{n})\leq \varepsilon /2,\text{ }
\end{equation*}%
\begin{equation*}
\left\langle y_{n}^{\ast },y_{n}-y\right\rangle \leq \varepsilon /2,\text{
and }g(x)-g(y_{n})\leq \varepsilon /2.
\end{equation*}%
Hence, for any $y\in \limfunc{dom}f$,%
\begin{eqnarray*}
\left\langle x_{n}^{\ast },y-x\right\rangle &=&\left\langle x_{n}^{\ast
},y-x_{n}\right\rangle +\left\langle x_{n}^{\ast },x_{n}-x\right\rangle \\
&\leq &f(y)-f(x_{n})+\varepsilon /2\leq f(y)-f(x)+\varepsilon ,
\end{eqnarray*}%
and we deduce\ that $x_{n}^{\ast }\in \partial _{\varepsilon }f(x).$
Similarly, $y_{n}^{\ast }\in \partial _{\varepsilon }g(x)\ $and, so, 
\begin{equation*}
x^{\ast }\in \partial _{\varepsilon }f(x)+\partial _{\varepsilon
}g(x)+\varepsilon B_{X^{\ast }},
\end{equation*}%
entailing\ 
\begin{equation*}
x^{\ast }\in \partial _{\varepsilon }\overline{f}^{w^{\ast \ast
}}(x)+\partial _{\varepsilon }\overline{g}^{w^{\ast \ast }}(x)+\varepsilon
B_{X^{\ast }}\subset \partial _{2\varepsilon }\left( \overline{f}^{w^{\ast
\ast }}+\overline{g}^{w^{\ast \ast }}\right) (x)+\varepsilon B_{X^{\ast }}.
\end{equation*}%
Thus, for each\ $\delta >0$, 
\begin{equation*}
\partial (f+g)(x)\subset \partial _{\delta }\left( \overline{f}^{w^{\ast
\ast }}+\overline{g}^{w^{\ast \ast }}\right) (x)+\varepsilon B_{X^{\ast }},%
\text{ for all }\varepsilon \in (0,\delta /2),
\end{equation*}%
and we obtain\ 
\begin{eqnarray*}
\partial (f+g)(x) &\subset &\tbigcap_{0<\varepsilon <\delta /2}\left(
\partial _{\delta }\left( \overline{f}^{w^{\ast \ast }}+\overline{g}%
^{w^{\ast \ast }}\right) (x)+\varepsilon B_{X^{\ast }}\right) \\
&=&\limfunc{cl}\nolimits^{\left\Vert \cdot \right\Vert _{\ast }}\left(
\partial _{\delta }\left( \overline{f}^{w^{\ast \ast }}+\overline{g}%
^{w^{\ast \ast }}\right) (x)\right) =\partial _{\delta }\left( \overline{f}%
^{w^{\ast \ast }}+\overline{g}^{w^{\ast \ast }}\right) (x),
\end{eqnarray*}%
implying 
\begin{equation*}
\partial (f+g)(x)\subset \tbigcap_{\delta >0}\partial _{\delta }\left( 
\overline{f}^{w^{\ast \ast }}+\overline{g}^{w^{\ast \ast }}\right)
(x)=\partial \left( \overline{f}^{w^{\ast \ast }}+\overline{g}^{w^{\ast \ast
}}\right) (x).
\end{equation*}%
Now, applying\ Lemma \ref{lembb} successively to $f+g$ and $\overline{f}%
^{w^{\ast \ast }}+\overline{g}^{w^{\ast \ast }}$, we conclude that, for all $%
z\in X^{\ast \ast },$ 
\begin{eqnarray*}
\partial (\overline{f+g}^{w^{\ast \ast }})(z) &=&\tbigcap_{U\in \mathcal{N}}~%
\limfunc{cl}\nolimits^{\left\Vert \cdot \right\Vert _{\ast }}\left(
\tbigcup_{x\in z+U}\partial (f+g)(x)\right) \\
&\subset &\tbigcap_{U\in \mathcal{N}}~\limfunc{cl}\nolimits^{\left\Vert
\cdot \right\Vert _{\ast }}\left( \tbigcup_{x\in z+U}\partial \left( 
\overline{f}^{w^{\ast \ast }}+\overline{g}^{w^{\ast \ast }}\right) (x)\right)
\\
&=&\partial \left( \overline{f}^{w^{\ast \ast }}+\overline{g}^{w^{\ast \ast
}}\right) (z).
\end{eqnarray*}%
Consequently, by applying Rockafellar's theorem (see Proposition \ref{roctem}
below) in the Banach space $X^{\ast \ast },$ the proper $w^{\ast \ast }$-lsc
(hence, $\left\Vert \cdot \right\Vert _{\ast \ast }$-lsc) convex functions $%
\overline{f+g}^{w^{\ast \ast }}$ and\ $\overline{f}^{w^{\ast \ast }}+%
\overline{g}^{w^{\ast \ast }}$ coincide up to some additive constant.$\ $%
Thus, (\ref{equalityco}) follows as these two functions coincide with $f+g$
on $X.$

$(v)\implies (vi):$ Fix $x\in X,$\ $x^{\ast }\in \partial (f+g)(x),$ and
choose any sequence $\varepsilon _{n}\downarrow 0.$ Then, taking into
account Theorem \ref{prop1}$(iv)$ (applied with $\tau =\left\Vert \cdot
\right\Vert _{\ast }$), for each\ fixed $n$ assertion $(v)$ implies that 
\begin{equation*}
x^{\ast }\in \partial _{\frac{(\varepsilon _{n})^{2}}{4}}(f+g)(x)\subset
\partial _{\frac{(\varepsilon _{n})^{2}}{2}}f(x)+\partial _{\frac{%
(\varepsilon _{n})^{2}}{2}}g(x)+(\varepsilon _{n})^{2}B_{X^{\ast }};
\end{equation*}%
that is, there are $u_{n}^{\ast }\in \partial _{\frac{(\varepsilon _{n})^{2}%
}{2}}f(x)$ and $v_{n}^{\ast }\in \partial _{\frac{(\varepsilon _{n})^{2}}{2}%
}g(x)$ such that 
\begin{equation}
\left\Vert x^{\ast }-u_{n}^{\ast }-v_{n}^{\ast }\right\Vert \leq
(\varepsilon _{n})^{2}.  \label{by}
\end{equation}%
We introduce\ the function $\varphi \in \Gamma _{0}(X\times X)$ defined as $%
\varphi (y,z):=f(y)+g(z).$ Since it can be easily check that\ $(u_{n}^{\ast
},v_{n}^{\ast })\in \partial _{(\varepsilon _{n})^{2}}\varphi (x,x),$\ by
applying the Brondsted-Rockafellar\ theorem (see, e.g., \cite[Proposition
4.3.7]{CHL23}) in the Banach product space $X\times X$ (endowed with max\
norm), we find\ $x_{n},$ $y_{n}\in B_{\varepsilon _{n}}(x),$ $\alpha _{n}\in
\lbrack -1,1],$ $y_{n}^{\ast },$ $z_{n}^{\ast }\in B_{X^{\ast }},$ $%
x_{n}^{\ast }\in \partial f(x_{n})$ and $y_{n}^{\ast }\in \partial g(y_{n})$
such that 
\begin{equation}
\left\vert \left\langle x_{n}^{\ast },x_{n}-x\right\rangle +\left\langle
y_{n}^{\ast },y_{n}-x\right\rangle \right\vert \leq \varepsilon
_{n}^{2}+\varepsilon _{n},  \label{v1}
\end{equation}%
\begin{equation}
\left\vert \varphi (x_{n},y_{n})-\varphi (x,x)\right\vert =\left\vert
f(x_{n})-f(x)+g(y_{n})-g(x)\right\vert \leq \varepsilon _{n}^{2}+\varepsilon
_{n},  \label{v2}
\end{equation}%
\begin{equation*}
x_{n}^{\ast }=u_{n}^{\ast }+\varepsilon _{n}(y_{n}^{\ast }+\alpha
_{n}u_{n}^{\ast })\text{ and }y_{n}^{\ast }=v_{n}^{\ast }+\varepsilon
_{n}(z_{n}^{\ast }+\alpha _{n}v_{n}^{\ast }).
\end{equation*}%
Thus, by (\ref{by}) and the triangle inequality, 
\begin{eqnarray*}
\left\Vert x^{\ast }-x_{n}^{\ast }-y_{n}^{\ast }\right\Vert &=&\left\Vert
x^{\ast }-u_{n}^{\ast }-v_{n}^{\ast }-\varepsilon _{n}(y_{n}^{\ast
}+z_{n}^{\ast }+\alpha _{n}(u_{n}^{\ast }+v_{n}^{\ast }))\right\Vert \\
&\leq &(\varepsilon _{n})^{2}+\varepsilon _{n}(\left\Vert y_{n}^{\ast
}+z_{n}^{\ast }\right\Vert +\left\vert \alpha _{n}\right\vert \left\Vert
u_{n}^{\ast }+v_{n}^{\ast }\right\Vert ) \\
&\leq &(\varepsilon _{n})^{2}+\varepsilon _{n}(2+\left\vert \alpha
_{n}\right\vert (\varepsilon _{n})^{2}+\left\vert \alpha _{n}\right\vert
\left\Vert x^{\ast }\right\Vert ) \\
&\leq &(\varepsilon _{n})^{2}+\varepsilon _{n}(2+(\varepsilon
_{n})^{2}+\left\Vert x^{\ast }\right\Vert ),
\end{eqnarray*}%
and we deduce\ that $x_{n}^{\ast }+y_{n}^{\ast }\overset{\left\Vert \cdot
\right\Vert _{\ast }}{\rightarrow }x^{\ast }.$ Moreover, thanks to (\ref{v2}%
), the lower semicontinuity of $f$ and $g$ implies that $f(x_{n})\rightarrow
f(x)$ and $g(y_{n})\rightarrow g(x).$ At the same time, since\ 
\begin{equation*}
\left\langle x_{n}^{\ast },x_{n}-x\right\rangle \leq f(x_{n})-f(x),
\end{equation*}%
and 
\begin{equation*}
\left\langle x_{n}^{\ast },x_{n}-x\right\rangle \geq \left( \left\langle
x_{n}^{\ast },x_{n}-x\right\rangle +\left\langle y_{n}^{\ast
},y_{n}-x\right\rangle \right) +g(x)-g(y_{n}),
\end{equation*}%
the inequality in (\ref{v1}) yields $\left\langle x_{n}^{\ast
},x_{n}-x\right\rangle \rightarrow 0,$ as well as\ $\left\langle y_{n}^{\ast
},y_{n}-x\right\rangle \rightarrow 0.$ So, the sequences $(x_{n}),$ $(y_{n})$%
, $(x_{n}^{\ast })$ and $(y_{n}^{\ast })$ satisfy the properties required in
assertion $(vi)$.
\end{dem}

When $X$ is a reflexive Banach space, equality (\ref{equalityco}) is always
satisfied when the involved functions are in $\Gamma _{0}(X).$ So, the
calculus rule in $(vi)$ is always valid in this setting, and we recover the
results established in \cite{T1} and \cite{T2}; see, also, \cite{CHJ16} for
related results.

Next, we show that usual classical qualifications ensure the validity of (%
\ref{equalityco}) (in the convex case). The following\ corollary follows\
from Theorem \ref{prop1}.

\begin{cor}
\label{corqualif}Let $f,$ $g:X\rightarrow \mathbb{R}_{\infty }$ be convex\
functions such that $\limfunc{dom}f\cap \limfunc{dom}g\neq \emptyset ,$ each
one having\ a continuous\ affine minorant. Then (\ref{equalityco}) holds
provided that $f$ is continuous at some point in $\limfunc{dom}g$ (or $X$
and the functions $f$ and $g$ satisfy any condition in \cite[Theorem 2.8.7]%
{Za02}).
\end{cor}

\begin{dem}
Under the current\ conditions, for every $x\in \limfunc{dom}f\cap \limfunc{%
dom}g$ and $\varepsilon \geq 0$ we have the following\ (exact) calculus rule
(see, e.g., \cite[Theorem 2.8.7]{CHL23} and \cite[Theorem 2.8.7]{Za02}): 
\begin{equation*}
\partial _{\varepsilon }(f+g)(x)=\tbigcup_{\substack{ \varepsilon
_{1}+\varepsilon _{2}=\varepsilon  \\ \varepsilon _{1},\varepsilon _{2}\geq
0 }}\partial _{\varepsilon _{1}}f(x)+\partial _{\varepsilon _{2}}g(x).
\end{equation*}%
So, all the conditions of Theorem \ref{prop1} are satisfied.
\end{dem}

\begin{rem}
Equality (\ref{equalityco}) also fulfills when the functions $f$ and $g$ are
lsc and convex,\ and the epigraph of $f^{\ast }\square g^{\ast }$ is $\sigma
(X^{\ast },X)$-closed (\cite{JeBu05}).
\end{rem}

Remember that for any pair\ of nonempty convex subsets $A$ and $B$ of $X$, a
classical result states that 
\begin{equation*}
\overline{A\cap B}=\overline{A}\cap \overline{B},
\end{equation*}%
provided that $A\cap (\limfunc{int}B)\neq \emptyset .$ The following
corollary shows that\ the relation above also holds for the $w^{\ast \ast }$%
-closure in $X^{\ast \ast }$, under the same interiority condition\
verified\ in $X$.

\begin{cor}
Let $A,$ $B$ be\ nonempty convex sets such that $A\cap (\limfunc{int}B)\neq
\emptyset .$ Then 
\begin{equation*}
\overline{A\cap B}^{w^{\ast \ast }}=\overline{A}^{w^{\ast \ast }}\cap 
\overline{B}^{w^{\ast \ast }}.
\end{equation*}
\end{cor}

\begin{dem}
According to Corollary \ref{corqualif}, we have that 
\begin{equation*}
\mathrm{I}_{\overline{A\cap B}^{w^{\ast \ast }}}=\overline{\mathrm{I}_{A\cap
B}}^{w^{\ast \ast }}=\overline{\mathrm{I}_{A}+\mathrm{I}_{B}}^{w^{\ast \ast
}}=\mathrm{I}_{\overline{A}^{w^{\ast \ast }}}+\mathrm{I}_{\overline{B}%
^{w^{\ast \ast }}}=\mathrm{I}_{\overline{A}^{w^{\ast \ast }}\cap \overline{B}%
^{w^{\ast \ast }}}.
\end{equation*}
\end{dem}

We give a functional counterpart of the previous observation, showing that
the usual interiority conditions applied to the original functions also give
rise\ to exact subdifferential calculus rules for the regularized functions.

\begin{theo}
\label{qualif}Let $f,$ $g:X\rightarrow \mathbb{R}_{\infty }$ be convex
functions with continuous\ affine minorants. If $f$ is continuous somewhere
in $\limfunc{dom}g$, then, for all $x\in X^{\ast \ast }$ and $\varepsilon
\geq 0,$\ 
\begin{eqnarray}
\partial _{\varepsilon }\left( \overline{f+g}^{w^{\ast \ast }}\right) (x)
&=&\partial _{\varepsilon }\left( \bar{f}^{w^{\ast \ast }}+\bar{g}^{w^{\ast
\ast }}\right) (x)  \notag \\
&=&\bigcup\limits_{\substack{ \varepsilon _{1},\varepsilon _{2}\geq 0  \\ %
\varepsilon _{1}+\varepsilon _{2}=\varepsilon }}\left( \partial
_{\varepsilon _{1}}\bar{f}^{w^{\ast \ast }}(x)+\partial _{\varepsilon _{2}}%
\bar{g}^{w^{\ast \ast }}(x)\right) .  \label{dd}
\end{eqnarray}
\end{theo}

\begin{proof}
The first equality is an immediate consequence of Corollary \ref{corqualif}.
To prove the second one we observe first that $\overline{f}^{w^{\ast \ast
}}, $ $\overline{g}^{w^{\ast \ast }}\in \Gamma _{0}(X^{\ast \ast }),$ due to
existence of continuous affine minorants for\ $f$ and $g$ in $X.$ Then,
thanks to Theorem \ref{prop1}, the functions $\overline{f}^{w^{\ast \ast }}$
and $\overline{g}^{w^{\ast \ast }}$ satisfy (\ref{sum1b});\ i.e., for all $%
x\in X^{\ast \ast }$ and $\varepsilon \geq 0,$%
\begin{equation}
\partial _{\varepsilon }\left( \overline{f}^{w^{\ast \ast }}+\overline{g}%
^{w^{\ast \ast }}\right) (x)=\tbigcap_{\alpha >\varepsilon }\limfunc{cl}%
\nolimits^{\tau }\left( A_{\alpha }\right) ,  \label{mas}
\end{equation}%
where 
\begin{equation*}
A_{\alpha }:=\bigcup\limits_{\substack{ \varepsilon _{1},\varepsilon
_{2}\geq 0  \\ \varepsilon _{1}+\varepsilon _{2}=\alpha }}\left( \partial
_{\varepsilon _{1}}\overline{f}^{w^{\ast \ast }}(x)+\partial _{\varepsilon
_{2}}\overline{g}^{w^{\ast \ast }}(x)\right) .
\end{equation*}%
Take $x^{\ast }\in \partial _{\varepsilon }\left( \overline{f}^{w^{\ast \ast
}}+\overline{g}^{w^{\ast \ast }}\right) (x)\ $and consider any net $(\alpha
_{V})_{V\in \mathcal{N}}\subset ]0,1[$ such that $\alpha _{V}\downarrow
\varepsilon $ (here, $\mathcal{N}:=\mathcal{N}_{X^{\ast }}(\tau )$ is
partially-ordered by descending inclusions). Hence, for all $V\in \mathcal{N}%
,$ (\ref{mas}) implies that $x^{\ast }\in \limfunc{cl}\nolimits^{\tau
}\left( A_{\alpha _{V}}\right) $ and we can find $x_{V}^{\ast }\in A_{\alpha
_{V}}$ such that $x_{V}^{\ast }\in x^{\ast }+V.$ Given any\ $W\in \mathcal{N}%
,$ by taking\ $V_{0}=W$ we see that 
\begin{equation*}
x_{V}^{\ast }\in x^{\ast }+V\subset x^{\ast }+W,\text{ for all }V\subset
V_{0};
\end{equation*}%
that is, $(x_{V}^{\ast })$ is $\tau $-convergent to $x^{\ast }.$ Moreover,
for each $V\in \mathcal{N}$ there are $\varepsilon _{V,1},$ $\varepsilon
_{V,2}\geq 0,$ $y_{V}^{\ast }\in \partial _{\varepsilon _{V,1}}\overline{f}%
^{w^{\ast \ast }}(x)$ and $z_{V}^{\ast }\in \partial _{\varepsilon _{V,2}}%
\overline{g}^{w^{\ast \ast }}(x)$ such that $\varepsilon _{V,1}+\varepsilon
_{V,2}=\alpha _{V}$ and $x_{V}^{\ast }=y_{V}^{\ast }+z_{V}^{\ast }.$ By the
current continuity hypothesis, we choose $x_{0}\in \limfunc{dom}g$ and $%
\theta $-neighborhood $U_{0}\subset X$ such that 
\begin{equation*}
\overline{f}^{w^{\ast \ast }}(x_{0}+u)\leq f(x_{0}+u)\leq f(x_{0})+1,\text{
for all }u\in U_{0}.
\end{equation*}%
Then, for all $u\in U_{0}$ ($\subset X$), 
\begin{eqnarray}
\left\langle y_{V}^{\ast },x_{0}-x+u\right\rangle &\leq &\overline{f}%
^{w^{\ast \ast }}(x_{0}+u)-\overline{f}^{w^{\ast \ast }}(x)+\varepsilon
_{V,1}  \notag \\
&\leq &f(x_{0})-\overline{f}^{w^{\ast \ast }}(x)+2.  \label{hamd}
\end{eqnarray}%
At the same time, there exists some $M\geq 0$ such that 
\begin{equation*}
\left\langle y_{V}^{\ast },x_{0}-x\right\rangle \geq -M,\text{ for all }V;
\end{equation*}%
otherwise, since $\left\langle x_{V}^{\ast },x_{0}-x\right\rangle
=\left\langle y_{V}^{\ast }+z_{V}^{\ast },x_{0}-x\right\rangle $ converges
to $\left\langle x^{\ast },x_{0}-x\right\rangle $ and $z_{V}^{\ast }\in
\partial _{\varepsilon _{V,2}}\overline{g}^{w^{\ast \ast }}(x),$ we get a
contradiction with 
\begin{equation*}
\left\langle z_{V}^{\ast },x_{0}-x\right\rangle \leq \overline{g}^{w^{\ast
\ast }}(x_{0})-\overline{g}^{w^{\ast \ast }}(x)+1,\text{ for all }V.
\end{equation*}%
Therefore (\ref{hamd}) reads\ 
\begin{equation*}
\left\langle y_{V}^{\ast },u\right\rangle \leq f(x_{0})-\overline{f}%
^{w^{\ast \ast }}(x)+M+2,\text{ for all }u\in U_{0},
\end{equation*}%
and,\ thanks to the Alaoglu-Bourbaki theorem, we may assume that $%
(y_{V}^{\ast })_{V}$ is $\sigma (X^{\ast },X)$-convergent to some $y^{\ast
}\in X^{\ast };$ hence, $(z_{V}^{\ast })_{V}$ is $\sigma (X^{\ast },X)$%
-convergent to $x^{\ast }-y^{\ast }.$ We can also assume that $\varepsilon
_{V,1}\downarrow \varepsilon _{1}$ and $\varepsilon _{V,2}\downarrow
\varepsilon _{2}$ with $\varepsilon _{1}+\varepsilon _{2}=\varepsilon .$
More precisely, since for all $y\in X$ and\ $V\in \mathcal{N}$ we have 
\begin{equation*}
\left\langle y_{V}^{\ast },y-x\right\rangle \leq \overline{f}^{w^{\ast \ast
}}(y)-\overline{f}^{w^{\ast \ast }}(x)+\varepsilon _{V,1}\leq f(y)-\overline{%
f}^{w^{\ast \ast }}(x)+\varepsilon _{V,1},
\end{equation*}%
by taking the limit on $V$ we infer that 
\begin{equation*}
\left\langle y^{\ast },y-x\right\rangle \leq f(y)-\overline{f}^{w^{\ast \ast
}}(x)+\varepsilon _{1}.
\end{equation*}%
Thus, knowing that $X$ is $\sigma (X^{\ast \ast },X^{\ast })$-dense in $%
X^{\ast \ast },$ we deduce that 
\begin{equation*}
\left\langle y^{\ast },y-x\right\rangle \leq \overline{f}^{w^{\ast \ast
}}(y)-\overline{f}^{w^{\ast \ast }}(x)+\varepsilon _{1},\text{ for all }y\in
X^{\ast \ast };
\end{equation*}%
that is, $y^{\ast }\in \partial _{\varepsilon _{1}}\overline{f}^{w^{\ast
\ast }}(x).$ Similarly, we show that $x^{\ast }-y^{\ast }\in \partial
_{\varepsilon _{2}}\overline{g}^{w^{\ast \ast }}(x)$, and we conclude that 
\begin{equation*}
x^{\ast }=y^{\ast }+(x^{\ast }-y^{\ast })\in \partial _{\varepsilon _{1}}%
\overline{f}^{w^{\ast \ast }}(x)+\partial _{\varepsilon _{2}}\overline{g}%
^{w^{\ast \ast }}(x).
\end{equation*}%
Thus, we are done because the inclusion \textquotedblleft $\supset $%
\textquotedblright\ in (\ref{dd}) is straightforward.
\end{proof}

\section{Illustration\label{Appl}}

To illustrate our previous results, we give the following proposition which
offers\ an alternative proof of the well-known integration result of the
Fenchel subdifferential of convex functions. Our\ proof relies on Corollary %
\ref{qualif} and Lemma \ref{lembb}, which were both established
independently of the present integration\ result. It is worth noting that
while this result was used in the proof of Theorem \ref{thmseq}, the latter
does not enter into the following proof.

\begin{prop}
\label{roctem} Let $f,$ $g\in \Gamma _{0}(X)$ be such that 
\begin{equation*}
\partial f(x)\subset \partial g(x),\text{ for all }x\in X.
\end{equation*}%
Then $f$ and $g$ coincide up to some additive constant, provided that one of
the following conditions holds:

$(i)$ $f$ is continuous somewhere in $X.$

$(ii)$ $X$ is Banach.
\end{prop}

\begin{proof}
$(i)$ Assume that $f$ is continuous at $x_{0}\in \limfunc{dom}f.$ Given $%
x\in X$, we introduce the functions $\varphi ,$ $\psi \in \Gamma _{0}(%
\mathbb{R})$ defined as 
\begin{equation*}
\varphi (s):=f(x_{0}+s(x-x_{0})),\text{ }\psi (s):=g(x_{0}+s(x-x_{0}),\text{ 
}s\in \mathbb{R}.
\end{equation*}%
Since\ $0\in \limfunc{int}(\limfunc{dom}\psi ),$ \cite[Lemma 4.4.2]{CHL23}
entails 
\begin{equation*}
\varphi (s)-\psi (s)=f(x_{0})-g(x_{0})=:c_{0},\text{ for all }s\in \mathbb{R}%
,
\end{equation*}%
which in turn implies $f(x)=g(x)+c_{0}.$

$(ii)$ Assume that $X$ is Banach and fix $x_{0}\in \limfunc{dom}f\cap 
\limfunc{dom}g$ together with $r>\left\Vert x_{0}\right\Vert .$ Given\ $z\in
X^{\ast \ast },$\ by Theorem\ \ref{qualif}\ (applied twice) and Lemma \ref%
{lembb}, we have that 
\begin{eqnarray*}
\partial (\bar{f}^{w^{\ast \ast }}+\mathrm{I}_{rB_{X^{\ast \ast }}})(z)
&=&\partial (\overline{f+\mathrm{I}_{rB_{X}}}^{w^{\ast \ast }})(z) \\
&=&\tbigcap_{U\in \mathcal{N}}\limfunc{cl}\nolimits^{\left\Vert \cdot
\right\Vert _{\ast }}\left( \tbigcup_{x\in z+U}\partial (f+\mathrm{I}%
_{rB_{X}})(x)\right) \\
&=&\tbigcap_{U\in \mathcal{N}}\limfunc{cl}\nolimits^{\left\Vert \cdot
\right\Vert _{\ast }}\left( \tbigcup_{x\in z+U}\left( \partial f(x)+\mathrm{N%
}_{rB_{X}}(x)\right) \right) ,
\end{eqnarray*}%
where $\mathcal{N}:=\mathcal{N}_{X^{\ast \ast }}(w^{\ast \ast }).$ Hence,
using the current assumption, 
\begin{equation*}
\partial (\bar{f}^{w^{\ast \ast }}+\mathrm{I}_{rB_{X^{\ast \ast
}}})(z)\subset \tbigcap_{U\in \mathcal{N}}\limfunc{cl}\nolimits^{\left\Vert
\cdot \right\Vert _{\ast }}\left( \tbigcup_{x\in z+U}\left( \partial g(x)+%
\mathrm{N}_{rB_{X}}(x)\right) \right) ,
\end{equation*}%
and, proceeding as above, we conclude that 
\begin{equation*}
\partial (\bar{f}^{w^{\ast \ast }}+\mathrm{I}_{rB_{X^{\ast \ast
}}})(z)\subset \partial (\bar{g}^{w^{\ast \ast }}+\mathrm{I}_{rB_{X^{\ast
\ast }}})(z),\text{ for all }z\in X^{\ast \ast }.
\end{equation*}%
Moreover, using Moreau's theorem, we have that 
\begin{equation*}
(\bar{f}^{w^{\ast \ast }}+\mathrm{I}_{rB_{X^{\ast \ast }}})^{\ast }=(%
\overline{f+\mathrm{I}_{rB_{X}}}^{w^{\ast \ast }})^{\ast }=(f+\mathrm{I}%
_{rB_{X}})^{\ast }=f^{\ast }\square (r\left\Vert \cdot \right\Vert _{\ast
})=:f_{r},
\end{equation*}%
and similarly $(\bar{g}^{w^{\ast \ast }}+\mathrm{I}_{rB_{X^{\ast \ast
}}})^{\ast }=g^{\ast }\square (r\left\Vert \cdot \right\Vert _{\ast
})=:g_{r};$ hence, for all $x^{\ast }\in X^{\ast },$%
\begin{equation*}
\partial f_{r}(x^{\ast })=(\partial (\bar{f}^{w^{\ast \ast }}+\mathrm{I}%
_{B_{X^{\ast \ast }}}))^{-1}(x^{\ast })\subset (\partial (\bar{g}^{w^{\ast
\ast }}+\mathrm{I}_{B_{X^{\ast \ast }}}))^{-1}(x^{\ast })=\partial
g_{r}(x^{\ast }).
\end{equation*}%
Therefore, since $f_{r}$ is (norm-)continuous, by $(i)$ there exists some $%
c_{r}\in \mathbb{R}$ such that 
\begin{equation*}
f_{r}=g_{r}-c_{r}.
\end{equation*}%
So, taking the conjugate with respect to the pair $(X,X^{\ast })$ and using
Moreau's theorem, we obtain 
\begin{equation*}
f+\mathrm{I}_{rB_{X}}=g+\mathrm{I}_{rB_{X}}+c_{r}.
\end{equation*}%
In particular, $c_{r}=f(x_{0})-g(x_{0})=:c_{0}$ and we deduce that $f+%
\mathrm{I}_{rB_{X}}=g+\mathrm{I}_{rB_{X}}+c_{0},$ for all sufficiently large 
$r;$ that is, $f=g+c_{0}.$
\end{proof}

\begin{rem}
It is worth noting that statement $(i)$ in Proposition \ref{roctem} can also
be obtained from $(ii)$. Indeed, given any finite-dimensional linear
subspace $L\subset X$ that intersects the set of continuity points of $f,$
we obtain, for all $x\in X,$\ 
\begin{equation*}
\partial (f+\mathrm{I}_{L})(x)=\partial f(x)+\partial \mathrm{I}%
_{L}(x)\subset \partial g(x)+\partial \mathrm{I}_{L}(x)\subset \partial (g+%
\mathrm{I}_{L})(x).
\end{equation*}%
So, it suffices to apply statement $(ii)$ in Proposition \ref{roctem} to the
restriction of the function $f+\mathrm{I}_{L}$ to $L$ (which is obviously
Banach).\medskip
\end{rem}

\end{document}